\documentclass[a4paper,12pt]{amsart}
\usepackage{amsmath,amssymb, amscd,latexsym,here,comment,color}

\newtheorem{theorem}{Theorem}

\newtheorem{Example}{Example}

\newtheorem{lemma}{Lemma}
\newtheorem{corollary}{Corollary}
\newtheorem{assertion}{Assertion}
\newtheorem{remark}{Remark}
\newtheorem{Transversalitytheorem}[theorem]{Transversality Theorem}
\newcommand{\zz}{\mathbf z}

\newcommand{\ww}{\mathbf w}
\newcommand\inv{^{-1}}

\newcommand{\si}{\sigma}
\def\mapright#1{\smash{\mathop{\longrightarrow}\limits^{{#1}}}}
\def\mapdown#1{\Big\downarrow\rlap{$\vcenter{\hbox{$#1$}}$}}

\title[Topology of mixed hypersurfaces of cyclic type]{Topology of mixed hypersurfaces of cyclic type}
\author{Kazumasa Inaba, Masayuki Kawashima and Mutsuo Oka}
\address[K. Inaba]{Mathematical Institute, Tohoku University, Sendai, 980-8578, Japan}
\email{sb0d02@math.tohoku.ac.jp }
\address[M. Kawashima]{Department of Information Science, Okayama University of Science, 1-1 Ridai-cho, Kitaku, Okayama 
700-0005, Japan}
\email{kawashima@mis.ous.ac.jp}
\address[M. Oka]{Department of Mathematics, Tokyo University of Science, 1-3 Kagurazaka, Shinjuku-ku, Tokyo 
162-8601, Japan}
\email{oka@rs.kagu.tus.ac.jp}

\begin{document}
\renewcommand{\thefootnote}{\fnsymbol{footnote}}
\footnote[0]{2010\textit{ Mathematics Subject Classification}.
Primary 14J17; Secondary: 14J70, 58K05.}


\maketitle


\begin{abstract}
Let 
$f_{II}(\zz, \bar{\zz}) = z_{1}^{a_{1}+b_{1}}\bar{z}_{1}^{b_{1}}z_{2} + \cdots + z_{n-1}^{a_{n-1}+b_{n-1}}\bar{z}_{n-1}^{b_{n-1}}z_{n} + z_{n}^{a_{n}+b_{n}}\bar{z}_{n}^{b_{n}}z_{1}$ 
be a~mixed weighted homogeneous polynomial of cyclic type and 
$g_{II}(\zz) = z_{1}^{a_{1}}z_{2} + \cdots + z_{n-1}^{a_{n-1}}z_{n} + z_{n}^{a_{n}}z_{1}$ 
be the associated weighted homogeneous polynomial
in the sense of \cite{O1} where $a_{j} \geq 1$ and $b_{j} \geq 0$ for $j = 1, \dots, n$. 
We show that two links $S^{2n-1}_{r} \cap f_{II}^{-1}(0)$ 
and $S^{2n-1}_{r} \cap g_{II}^{-1}(0)$ are diffeomorphic and their Milnor fibrations are isomorphic. 
\end{abstract}


\section{Introduction}
Let $f(\zz, \bar{\zz})$ be a mixed polynomial 
of complex variables $\zz = (z_1, \dots, z_n)$ 
given as 
\[
f(\zz, \bar{\zz}) := \sum_{i = 1}^{m} c_{i}\zz^{\nu_{i}}\bar{\zz}^{\mu_{i}}, 
\]
where $c_{i} \in \Bbb{C}^*$ and 
$\zz^{\nu_{i}} = z^{\nu_{i,1}}_1 \cdots z^{\nu_{i,n}}_n$  
for $\nu_{i} = (\nu_{i,1}, \dots, \nu_{i,n})$  
(respectively $\bar{\zz}^{\mu_{i}} = \bar{z}_{1}^{\mu_{i,1}} \cdots \bar{z}_{n}^{\mu_{i,n}}$ for $\mu_{i} = (\mu_{i,1}, \dots, \mu_{i,n}))$. 
Here $\bar{z}_j$~represents the complex conjugate of $z_j$. 


A point $\ww \in \Bbb{C}^{n}$ is called \textit{a mixed singular point of $f(\zz, \bar{\zz})$} 
if~the gradient vectors of $\Re f$ and $\Im f$ are linearly dependent at~$\ww$. 
Certain restricted classes of mixed
polynomials of the variables $\zz$  which admit Milnor fibrations 
had been considered by J. Seade, 
see for instance \cite{S1, S2}. 
 The last author introduced the notion of the Newton boundary and the concept of non-degeneracy for a mixed polynomial
and  he showed  the existence of Milnor fibration  for  the  class of
strongly  non-degenerate  mixed polynomials  \cite{O2}.

We consider the classes of mixed polynomials 
which was first introduced by Ruas-Seade-Verjovsky \cite{RSV} and J. L. Cisneros-Molina \cite{C}. 
Let $p_{1}, \dots, p_{n}$ and $q_{1}, \dots, q_{n}$ be integers such that 
$\gcd(p_1, \dots, p_n) = \gcd(q_1,\dots,q_n) = 1$. 
We define the~$S^1$-action and the~$\Bbb{R}^{*}$-action on~$\Bbb{C}^{n}$ as follows: 
\begin{equation*}
\begin{split}
s \circ \zz &= (s^{p_{1}}z_{1}, \dots, s^{p_{n}}z_{n}), \ \ s \in S^{1}, \\
r \circ \zz &= (r^{q_1}z_{1}, \dots, r^{q_n}z_{n}), \ \ r \in \Bbb{R}^{*}. 
\end{split}
\end{equation*}
If there exists a positive integer $d_p$ such that $f(\zz, \bar{\zz})$ satisfies 
\[
f(s^{p_{1}}z_{1}, \dots, s^{p_{n}}z_{n}, \bar{s}^{p_1}\bar{z}_{1}, \dots, \bar{s}^{p_1}\bar{z}_{n})
 = s^{d_p}f(\zz, \bar{\zz}), \ \ s \in S^{1},  \notag \\
\]
we say that $f(\zz, \bar{\zz})$ is \textit{a polar weighted homogeneous polynomial}. 
Similarly $f(\zz, \bar{\zz})$ is called \textit{a~radial weighted homogeneous polynomial} 
if there exists a positive integer $d_{r}$ such that 
\[
f(r^{q_1}z_{1}, \dots, r^{q_n}z_{n}, r^{q_1}\bar{z}_{1}, \dots, r^{q_n}\bar{z}_{n}) 
= r^{d_{r}}f(\zz, \bar{\zz}), \ \ r \in \Bbb{R}^{*}. 
\]
Let $f$ be a polar and radial weighted homogeneous polynomial.  
Then $f$ admits 
the global Milnor fibration $f : \Bbb{C}^{n} \setminus f^{-1}(0) \rightarrow \Bbb{C}^{*}$,
see for instance \cite{RSV, C, O1, O2}. 

Let $f(\zz, \bar{\zz}) = \sum_{i = 1}^{m} c_{i}\zz^{\nu_{i}}\bar{\zz}^{\mu_{i}}$ be a mixed polynomial
with $c_j\ne 0, j=1,\dots, m$.
Put 
\[
g(\zz) := \sum_{i = 1}^{m} c_{i}\zz^{\nu_{i} - \mu_{i}}. 
\]
We call $g$ \textit{the associated Laurent polynomial of $f$}. 
A mixed polynomial $f$ is called \textit{simplicial} if  $m\le n$ and 
the ranks of 
the matrices $N \pm M$ are $m$
 where 
$N = (\nu_{1}, \dots, \nu_{n})$ and $M = (\mu_{1}, \dots, \mu_{n})$.  Here $\nu_i$ and $\mu_i$ are considered as column vectors
$\nu_i= {}^t(\nu_{i1},\dots,\nu_{in}),\,\mu_i= {}^t(\mu_{i1},\dots,\mu_{in})$.
$f$ is called  {\em full} 
if $m=n$.
A full simplicial mixed polynomial $f$ and its associated Laurent polynomial~$g$ 
admit a unique polar weight and a~unique radial weight in the above sense \cite{O1}. 
It is useful to consider a graph $\Gamma$ associated to $f$.
First we associate a vertex $v_i$ if $z_i$ or $\bar z_i$ appears in $f$.  We join $v_i$ and $v_j$  by an edge if there is a monomial
$\zz^{\nu_k}{\bar{\zz}}^{\mu_k}$ which contains both variables $z_i, z_j$.
That is $\nu_{k,a}+\mu_{k,a}>0$ for $a=i, j$. Most important graphs are a bamboo graph
\[
B_n:\,\overset{v_1}\bullet \rule[1mm]{1cm}{0.3mm}\overset{v_2}\bullet \rule[1mm]{1cm}{0.3mm}\dots
\rule[1mm]{1cm}{0.3mm}
\overset{v_{n-1}}\bullet\rule[1mm]{1cm}{0.3mm}\overset{v_n}\bullet
\]
and a cyclic graph $C_n$ which is obtained from $B_n$ adding an edge between $v_n$ and $v_1$.

We restrict the Milnor fibrations defined by $f$ and $g$ on the complex torus $\Bbb{C}^{*n}$ 
where $\Bbb{C}^{*n} = (\Bbb{C}^{*})^{n}$. 
In \cite[Theorem 10]{O1}, it is shown that there exists a canonical
diffeomorphism $\varphi : \Bbb{C}^{*n} \rightarrow \Bbb{C}^{*n}$ 
which gives an isomorphism of the Milnor fibrations defined by $f$ and $g$: 

\[\begin{matrix}
\Bbb{C}^{*n}\setminus f^{-1}(0)&\mapright{\varphi}& \Bbb{C}^{*n}\setminus g^{-1}(0) \\
\mapdown{f}&&\mapdown{g}\\
\Bbb{C}^{*}&=& \Bbb{C}^{*} 
\end{matrix}. 
\] 
However the canonical diffeomorphism $\varphi$ does not extend to $\mathbb C^n\setminus\{ O\}$ in general.  
Here $O$ is the origin of $\Bbb{C}^{n}$. 
The exceptional case is  a mixed Brieskorn polynomial,  for which this canonical diffeomorphism extends as  a continuous homeomorphism \cite{RSV}.
In \cite{O3},  the last author studied the following simplicial polar weighted homogeneous polynomials:  
\begin{equation*}
\begin{cases}
f_{\mathbf{a},\mathbf{b}}(\zz, \bar{\zz}) &= z_{1}^{a_{1}+b_{1}}\bar{z}_{1}^{b_{1}} + \cdots + z_{n-1}^{a_{n-1}+b_{n-1}}\bar{z}_{n-1}^{b_{n-1}} + z_{n}^{a_{n}+b_{n}}\bar{z}_{n}^{b_{n}} \\
f_{I}(\zz, \bar{\zz}) &= z_{1}^{a_{1}+b_{1}}\bar{z}_{1}^{b_{1}}z_{2} + \cdots + z_{n-1}^{a_{n-1}+b_{n-1}}\bar{z}_{n-1}^{b_{n-1}}z_{n} + z_{n}^{a_{n}+b_{n}}\bar{z}_{n}^{b_{n}} \\
f_{II}(\zz, \bar{\zz}) &= z_{1}^{a_{1}+b_{1}}\bar{z}_{1}^{b_{1}}z_{2} + \cdots + z_{n-1}^{a_{n-1}+b_{n-1}}\bar{z}_{n-1}^{b_{n-1}}z_{n} + z_{n}^{a_{n}+b_{n}}\bar{z}_{n}^{b_{n}}z_{1}, \\
\end{cases}
\end{equation*}
where $a_{j} \geq 1$ and $b_{j} \geq 0$ for $j =1, \dots ,n$. 
Here  the notation is the same  as  in \cite{O3}.  
Note that the graph of $f_I$ is a bamboo and that of $f_{II}$ is a cyclic graph. 
The graph of $f_{\mathbf{a},\mathbf{b}}$ is $n$~disjoint vertices without any edges.
A polar weighted homogeneous polynomial $f_{\mathbf{a},\mathbf{b}}(\zz, \bar{\zz})$ and 
a~weighted homogeneous polynomial $g_{\mathbf{a}}(\zz)$ are called 
\textit{a mixed Brieskorn polynomial} and 
\textit{a Brieskorn polynomial} respectively. 
A mixed polynomial~$f_{I}(\zz, \bar{\zz})$ is called  
\textit{a~simplicial mixed polynomial of bamboo type} and 
$f_{II}(\zz, \bar{\zz})$ is called \textit{a~simplicial mixed polynomial of cyclic type} respectively. 
He showed that  two links of $f_{\iota}$ and the associated polynomial $g_\iota({\bf  z})$ 
in a small sphere are isotopic and their Milnor fibrations are isomorphic 
for  $\iota = (\mathbf{a}, \mathbf{b})$ and $I$. 
He  conjectured the assertion will be also true for the case $f_{II}$.

\section{Statement of the result}
The purpose of this paper is to give a positive answer to the above  conjecture. Thus
we study the following simplicial polynomial $f_{II}(\zz, \bar{\zz})$ 
and its associated weighted homogeneous polynomial $g_{II}(\zz)$: 
\begin{equation*}
\begin{cases}
f_{II}(\zz, \bar{\zz}) &= z_{1}^{a_{1}+b_{1}}\bar{z}_{1}^{b_{1}}z_{2} + \cdots + z_{n-1}^{a_{n-1}+b_{n-1}}\bar{z}_{n-1}^{b_{n-1}}z_{n} + z_{n}^{a_{n}+b_{n}}\bar{z}_{n}^{b_{n}}z_{1}, \\
g_{II}(\zz) &= z_{1}^{a_{1}}z_{2} + \cdots + z_{n-1}^{a_{n-1}}z_{n} + z_{n}^{a_{n}}z_{1}, 
\end{cases}
\end{equation*}
where $a_{j} \geq 1$ and $b_{j} \geq 0$ for $j =1, \dots ,n$. 
We assume that $f_{II}$ contains a conjugate ${\bar z_j}$ for some $j$.
  This implies \newline
(a)  there exists $j\in \{1, \dots ,n\}$ such that $b_{j} \geq 1$. 

 We also assume that $f_{II}$ is simplicial.
As  the determinant of $N-M$ is given by $a_1\cdots a_n+(-1)^{n+1}$,
we assume also that 
\newline
 (b) there exists $k\in \{1, \dots ,n\}$ such that $a_{k} \geq 2$. 

\noindent
Though this assumption is not necessary if $n$ is odd, we assume (b) anyway. 
Since $f_{II}(\zz, \bar{\zz})$ is a polar and radial weighted homogeneous, 
$f_{II}(\zz, \bar{\zz})$ admits a global fibration 
\[
f_{II}: \Bbb{C}^{n} \setminus f_{II}^{-1}(0) \rightarrow \Bbb{C}^{*} 
\]
\cite{RSV, C, O1, O2}. 
The complex polynomial $g_{II}(\zz)$ is a weighted homogeneous polynomial 
with respect to the same polar weight of $f_{II}$ and $g_{II}$ with $n=3$ is  listed in the classification of weighted homogeneous surfaces in $\mathbb C^3$ with isolated singularity \cite{OW}. 
We consider the hypersurfaces 
\[V_f:=f_{II}\inv(0),\quad V_g:=g_{II}\inv(0)\]
 and respective links
 \[K_{f,r} =V_f \cap S^{2n-1}_{r},\quad K_{g,r} =V_g \cap S^{2n-1}_{r} \]
 where $S^{2n-1}_{r}$ is the~$(2n-1)$-dimensional sphere centered at the origin $O$ with radius $r$. 
Then 
two links $K_{f,r}$ and $K_{g,r}$ are smooth for any $r>0$  (\cite{O1}).
We consider the following family of mixed polynomials: 
\begin{equation*}
\begin{split}
f_{II, t}(\zz, \bar{\zz}) := &(1 - t)f_{II}(\zz, \bar{\zz}) + tg_{II}(\zz) \\
= &(1-t)(z_{1}^{a_{1}+b_{1}}\bar{z}_{1}^{b_{1}}z_{2} + \cdots + z_{n-1}^{a_{n-1}+b_{n-1}}\bar{z}_{n-1}^{b_{n-1}}z_{n} + z_{n}^{a_{n}+b_{n}}\bar{z}_{n}^{b_{n}}z_{1}) \\
&+ t(z_{1}^{a_{1}}z_{2} + \cdots + z_{n-1}^{a_{n-1}}z_{n} + z_{n}^{a_{n}}z_{1}) \\
= &\sum_{j=1}^{n}z_{j}^{a_{j}}z_{j+1}\{(1-t)\lvert z_{j} \rvert^{2b_{j}} + t\}
\end{split}
\end{equation*}
where $0 \leq t \leq 1$. 
Here the numbering is modulo $n$, so $z_{n+1} = z_{1}$. 
Though mixed polynomial 
$f_{II, t}$ is not radial weighted homogeneous for $t \neq 0, 1$,
$f_{II, t}$ is polar weighted homogeneous for $0 \leq t \leq 1$ with the same weight $P=(p_1,\dots, p_n)$
which is characterized by
$a_jp_j+p_{j+1}=d_p,\,j=1,\dots, n$.  
Put 
\[V_{t} = f_{II,t}^{-1}(0),\,\, K_{t, r} = S^{2n-1}_{r} \cap V_{t},\quad 0\le t\le 1.
\]
Note that 
\[\begin{split}
&f_{II,0}=f_{II},\,f_{II,1}=g_{II}\\
&V_{f}=V_{0},\,\, K_{f,r}=K_{0,r},\,\, V_{g}=V_{1},\,\, K_{g,r}=K_{1,r}.
\end{split}
\]
First recall that $V_{t}$ has an isolated mixed singularity at the origin $O$ 
and  $V_{t}\setminus\{ O \}$ is  non-singular for any $0\le t\le 1$ 
by \cite[Lemma 9]{O3}.  
Our main result is: 
\begin{Transversalitytheorem}\label{Transversality} Let $V_t$ be as above.
For any fixed $r>0$, the sphere $S^{2n-1}_{r}$ and the family of hypersurfaces~$V_{t}$ 
are transversal for $0 \le t \le 1$. 
\end{Transversalitytheorem}

\section{Proof of  Transversality Theorem \ref{Transversality}}
\subsection{Strategy of the proof}
We follow the recipe of  \cite{O3}.
First recall that
\begin{equation*}
\begin{split}
f_{II, t}(\zz, \bar{\zz}) := &(1 - t)f_{II}(\zz, \bar{\zz}) + tg_{II}(\zz) \\
= &\sum_{j=1}^{n}z_{j}^{a_{j}}z_{j+1}\{(1-t)\lvert z_{j} \rvert^{2b_{j}} + t\}.
\end{split}
\end{equation*}
Recall  that $V_t$ is non-singular off the origin by \cite{O3}.
To show the transversality of the sphere $S_{r_0}^{2n-1}$ and $V_t$, 
we have to show that the Jacobian matrix of $\Re f_{II,t},\Im f_{II,t}$ and $\rho(\bf z)$ has rank~$3$ 
at every intersection 
${\bf w}\in S_{r_0}^{2n-1}\cap V_t$. Here 
$\rho(\bf z)=\|\bf z\|^2$,  the square of the radius $\|\bf z\|$. However this computation is extremely complicated.
Instead, we follow the recipe of \cite{O3}. We will show {\em the existence of a tangent vector ${\bf v}\in T_{\ww}V_{t}$ which is  not tangent to the sphere $S_{r_0}^{2n-1}$.}

Take a point $\ww = (w_{1}, \dots, w_{n}) \in V_{t}\cap S_{r_0}^{2n-1}$ and fix it hereafter. 
To find such a vector~$\bf v$, 
we will construct a real analytic path 
\[
{\ww}(s) = (r_{1}(s)w_{1}, \dots, r_{n}(s)w_{n})
\]
\newline
on a neighborhood of $s=0$ so that $\ww(0)=\ww$ and 
\begin{eqnarray} \label{basicequality}
f_{II, t}(\ww(s), \bar{\ww}(s)) 
&=& (s+1) f_{II, t}(\ww, \bar{\ww}) \label{equality1}
\end{eqnarray}
where $r_{j}(s),\,j=1,\dots, n$ are  real-valued functions on  $\lvert s \rvert \ll 1$
which satisfy certain functional equalities.  
The equality (\ref{equality1}) implies that the curve $\ww(s)$ is  an embedded  curve in $V_{t}$ with  $\ww(0)=\ww$.
Then we define the vector as the tangent vector of this curve at $s=0$:
\begin{eqnarray}\label{equation2}
{\bf v}=\frac{d\ww}{ds}(0).
\end{eqnarray}
To find such a path
${\ww}(s)=(r_1(s)w_1,\dots,r_n(s)w_n)$, we solve  a certain functional equation, using the   inverse mapping  theorem. 
\subsection{Construction  of $\ww(s)$}
First we consider the following map: 
\begin{equation*}
\begin{split}
\Phi_{\ww}: \Bbb{R}^{n+1} &\rightarrow \Bbb{R}^{n+1} \\
(r_{1}, \dots, r_{n}, s) &\mapsto (h_{1}, \dots, h_{n}, s), 
\end{split}
\end{equation*}
where  $h_j$ is a polynomial function of variables $r_1,\dots, r_n$ and $s$ defined by
\begin{eqnarray}
h_{j} &= r_{j}^{a_{j}}r_{j+1}\{(1-t)\lvert w_{j}\rvert^{2b_{j}}r_{j}^{2b_{j}} + t \} - (s+1)\{(1-t)\lvert w_{j}\rvert^{2b_{j}} + t\},\, j = 1, \dots, n
\end{eqnarray}
where $t$ is fixed on $ 0 \le t \le 1$. 
We want to solve the equations $h_1=\dots=h_n=0$ in $r_1,\dots, r_n$ expressing $r_j$ as a function of $s$ so that we get 
the system of equations
 \begin{eqnarray}h_j(r_1(s),\dots, r_n(s),s)\equiv 0,\,\, j=1,\dots, n.
 \end{eqnarray} 
 This equality is equivalent to 
(\ref{basicequality}) which is more explicitly written as 
\begin{eqnarray*}
f_{II,t}(\ww(s),\bar {\ww}(s))&=& (s+1)f_{II,t}(\ww,\bar \ww)\quad\text{where}\notag\\
f_{II, t}(\ww(s), \bar{\ww}(s))&=&
\sum_{j=1}^{n}(r_{j}(s)w_j)^{a_{j}}(r_{j+1}(s)w_{j+1})
\{(1-t)\lvert w_{j}\rvert^{2b_{j}}r_{j}(s)^{2b_{j}} + t \},\\
(s+1)f_{II,t}(\ww,\bar{\ww})&=&(s+1)\sum_{j=1}^n w_j^{a_j}{w_{j+1}}\{(1-t)|w_j|^{2b_j}+t\}.
\end{eqnarray*}
We will solve the functional equality (\ref{basicequality}) using the inverse mapping theorem.
\begin{lemma}\label{l2}
The Jacobian matrix $J(\Phi_{\ww})$ has rank $n+1$ at $(r_1,\dots, r_n,s)=(1, \dots, 1,0)$,
 $0<t<1$ and 
 $r_0>0$
where
\[
J(\Phi_{\ww})=\left (
\begin{matrix}
\frac{\partial h_1}{\partial r_1}&\cdots&\frac{\partial h_1}{\partial r_{n}}&\frac{\partial h_1}{\partial s}\\
&\vdots&&\\
\frac{\partial h_{n}}{\partial r_1}&\cdots&\frac{\partial h_{n}}{\partial r_{n}}&\frac{\partial h_n}{\partial s}\\
\frac{\partial s}{\partial r_1}&\cdots&\frac{\partial s}{\partial r_n}&\frac{\partial s}{\partial s}
\end{matrix}
\right ).
\]
\end{lemma}

\begin{proof} By a direct computation,
the Jacobian matrix of $\Phi_{\ww}$ is given as 
\[
J(\Phi_{\ww})=
\left(
        \begin{array}{@{\,}cccccccc@{\,}}
\alpha_{1,1} & \alpha_{1,2} & 0 & \ldots & 0  & -\beta_{1} \\
0 & \alpha_{2,2} & \alpha_{2,3} & \ddots & \vdots & \vdots \\ 
\vdots & \ddots & \ddots & \ddots & \vdots  & \vdots \\
0 & \ldots & 0 & \alpha_{n-1,n-1} & \alpha_{n-1,n} & \vdots \\
\alpha_{n,1} & 0 & \ldots & 0 & \alpha_{n,n} & -\beta_{n} \\
0 & \ldots & \ldots & \ldots & 0 & 1
\end{array}
       \right), 
\]
where 
\begin{eqnarray*}
\alpha_{j,j} &=& r_{j}^{a_{j}-1}r_{j+1}\{(1-t)\lvert w_{j}\rvert^{2b_{j}}(a_{j} + 2b_{j})r_{j}^{2b_{j}} + a_{j}t \}, \\
\alpha_{j,j+1} &=& r_{j}^{a_{j}}\{(1-t)\lvert w_{j}\rvert^{2b_{j}}r_{j}^{2b_{j}} + t \},\\
\beta_{j}& =& \{(1-t)\lvert w_{j}\rvert^{2b_{j}} + t \},
\qquad j = 1, \dots, n.
\end{eqnarray*}
Since $0 < t < 1$, 
$\alpha_{j,j}$ and $\alpha_{j,j+1}$ are positive real numbers for each $j = 1, \dots, n$. 
The determinant $\det J(\Phi_{\ww})$  is given  as 
\begin{eqnarray}
\det J(\Phi_{\ww})&=&
\alpha_{1,1}\dots \alpha_{n,n} + (-1)^{n+1}\alpha_{1,2}\cdots \alpha_{n-1,n}\alpha_{n,1}  \\
&= &\prod_{j=1}^{n}r_{j}^{a_{j}}\{(1-t)\lvert w_{j}\rvert^{2b_{j}}(a_{j} + 2b_{j})r_{j}^{2b_{j}} 
+ a_{j}t \} \notag\\
&&+ 
(-1)^{n+1}\prod_{j=1}^{n}r_{j}^{a_{j}}\{(1-t)\lvert w_{j}\rvert^{2b_{j}}r_{j}^{2b_{j}} + t \}. \notag
\end{eqnarray}
The proof of  Lemma 1 is reduced to the following assertion.
\begin{assertion} $\det J(\Phi_{\ww})>0$.
\end{assertion} 

{\em Proof.}
(i) If $n$ is an odd number,  $\det J(\Phi_{\ww})$ at $(1, \dots, 1, 0)$ is obviously positive. 
\newline
(ii) Suppose that $n$ is a positive even number. 
Consider 
\[\alpha_{j,j}':=r_{j}^{a_{j}}\{(1-t)\lvert w_{j}\rvert^{2b_{j}}(a_{j} + 2b_{j})r_{j}^{2b_{j}} + a_{j}t \}.
\]
Note that $\prod_{j=1}^n\alpha_{j,j}=\prod_{j=1}^n\alpha_{j,j}'$.
 We have
the following.
\[\begin{split}
\det J (\Phi_{\ww})&=\prod_{j=1}^n \alpha_{j,j}-\prod_{j=1}^n\alpha_{j,j+1}
=\prod_{j=1}^n \alpha_{j,j}'-\prod_{j=1}^n\alpha_{j,j+1}\\
\det J(\Phi_{\ww})&\ge 0\iff\prod_{j=1}^n\frac{\alpha_{j,j}'}{\alpha_{j,j+1} }\ge 1.
\end{split} 
\]
As $\alpha_{j,j}'\ge\alpha_{j,j+1}$, 
the equality takes place if $\alpha_{j,j}'=\alpha_{j,j+1}$ for $j = 1, \dots, n$.
We assume that $\alpha_{j,j}'=\alpha_{j,j+1}$ for any~$j, 1 \leq j \leq n$. Then 
\begin{equation*}
(1-t)\lvert w_{j}\rvert^{2b_{j}}(a_{j} + 2b_{j}) + a_{j}t = (1-t)\lvert w_{j}\rvert^{2b_{j}} + t, \,  j = 1, \dots, n
\end{equation*}
and this is the case if and only if $(w_{j}, a_{j}) = (0, 1)$ or $(a_{j}, b_{j}) = (1, 0)$. 
Thus the Jacobian of $\Phi_{\ww}$ at $(1, \dots, 1, 0)$ is equal to $0$ if and only if 
$(w_{j}, a_{j}) = (0, 1)$ or $(a_{j}, b_{j}) = (1, 0)$ for $j = 1, \dots, n$. 
 However this case does not happen,  as there exists $j$ with $a_j\ge 2$.  Thus
 the assertion is proved. This completes also the proof of Lemma 1.
\end{proof}
Now we are ready to prove the transversality of $S^{2n-1}_{r_0}$ and $V_{ t}$ for any $r_0>0$ and $0\le t\le 1$. 
\subsection{Proof of  Transversality Theorem}
The assertion is known for $t=0, 1$ by \cite{O2}. Thus  we assume that $0<t<1$.
Recall  that $f_{II, t} : \Bbb{C}^{n} \rightarrow \Bbb{C}$ has a unique singularity 
at the origin $O$ for any $0 \leq t \leq 1$ by 
\cite[Lemma 9]{O3}. 
As  the codimension of $T_{{\bf w}}S_{r_0}^{2n-1}$ in $\mathbb C^n\cong \mathbb R^{2n}$ is 1,
to show the transversality, it suffices to show the existence of a vector $\mathbf{v} \in T_{\ww}V_{ t}$ 
with $\mathbf{v}\notin T_{\ww}S^{2n-1}_{r_{0}}$. 

For a given $\ww = (w_{1}, \dots, w_{n}) \in V_{ t}$,  
we consider the nullity set $I_{\ww} = \{ i \mid w_{i} = 0\}$. 

\vspace{.2cm}
\noindent
{\bf Case $1. \ I_{\ww} = \emptyset$}. 
This  is the most essential case and  does not appear for the mixed polynomials 
$f_{\bf a,\bf b}$ and $f_I$. The corresponding graph is cyclic.

By Lemma \ref{l2}, 
the Jacobian of $\Phi_{\ww}$ at $(1,\dots, 1,0)$ is non-zero. 
By the Inverse mapping theorem,
there exist a neighborhood $U \subset \Bbb{R}^{n+1}$ of $(1,\dots, 1,0)$ and 
a neighborhood $W\subset \Bbb{R}^{n+1}$ of $\Phi_{\ww}(1,\dots, 1,0)=(0,\dots,0)$ and a real analytic mapping
\newcommand\id{\rm{id}}
$\Psi_{\ww}=(\psi_1,\dots, \psi_n,\id): W\to U$ so that 
\[\Phi_{\ww}\circ\Psi_{\ww}=\id_{W}\quad\text{and}\quad
\Psi_{\ww}\circ\Phi_{\ww}=\id_{U}.
\]
Put $\mathbf{0} :=(0,\dots,0)\in \Bbb{R}^{n}$ and consider 
 $V \subset \Bbb{R}:=W\cap(\{\bf 0\}\times \mathbb R)$, a neighborhood of $0\in \mathbb R$ and define 
smooth functions $r_{j} : V \rightarrow \Bbb{R}$ of the variable $s$  by $r_j(s):=\psi_j(0,\dots, 0,s)$. 
Note that $r_j(0)=1$. We have the equalities:
\[
h_{j}(r_{1}(s), \dots, r_{n}(s), s) \equiv 0,\, s\in V,\,j = 1, \dots, n. 
\]
As we have seen in the above discussion, this implies
\[\begin{split}
&r_{j}^{a_{j}}(s)r_{j+1}(s)\{(1-t)\lvert w_{j}\rvert^{2b_{j}}r_{j}(s)^{2b_{j}} + t \} - (s+1)\{(1-t)\lvert w_{j}\rvert^{2b_{j}} + t\}\equiv 0,\end{split}
\] which implies 
\[\begin{split}
f_{II, t}(\ww(s), \bar{\ww}(s))& =\sum_{j=1}^{n}r_j(s)^{a_j}r_{j+1}(s)w_{j}^{a_{j}}w_{j+1}
\{(1-t)\lvert w_{j}\rvert^{2b_{j}}r_{j}(s)^{2b_{j}} + t \}\\
&= (s+1)f_{II,t}(\ww,\bar{\ww}).
\end{split}
\]
 Thus  $f_{II,t}(\ww(s),\bar{\ww}(s))\equiv 0$.
 Put ${\bf v}=\frac{d\ww}{ds}(0)$.  We have
${\bf v}\in T_{\ww}V_{t}$ by the definition. 
Now to finish the proof of the transversality assertion,  we need only to show  
\begin{assertion}${\bf v}\ne {\bf 0}$ and ${\bf v}\notin T_{\ww}S_{r_0}^{2n-1}$.
\end{assertion}

To prove the assertion,  we consider the differential in $s$ of 
\begin{multline*}
h_{j}(r_{1}(s), \dots, r_{n}(s), s)= r_{j}(s)^{a_{j}}r_{j+1}(s)\{(1-t)\lvert w_{j}\rvert^{2b_{j}}r_{j}(s)^{2b_{j}}
 + t \} 
 - (s+1)\{(1-t)\lvert w_{j}\rvert^{2b_{j}} + t\}.
\end{multline*}
By a direct computation, 
we get the equality
\begin{equation*}
\begin{split}
\frac{d}{ds}h_{j}(r_{1}(s), \dots, r_{n}(s), s) 
&= \biggl( \sum_{k=1}^{n}\frac{\partial h_{j}}{\partial r_{k}}\frac{dr_{k}}{ds}\biggr) - \beta_{j} \\
&= \alpha_{j,j}\frac{dr_{j}}{ds} + \alpha_{j,j+1}\frac{dr_{j+1}}{ds} - \beta_{j} \equiv 0 
\end{split}
\end{equation*}
where
\[\begin{split}
\alpha_{j,j} &= r_{j}^{a_{j}-1}r_{j+1}\{(1-t)\lvert w_{j}\rvert^{2b_{j}}(a_{j} + 2b_{j})r_{j}^{2b_{j}} + a_{j}t \}, \\
\alpha_{j,j+1} &= r_{j}^{a_{j}}\{(1-t)\lvert w_{j}\rvert^{2b_{j}}r_{j}^{2b_{j}} + t \},\\
\beta_{j}& = \{(1-t)\lvert w_{j}\rvert^{2b_{j}} + t \}
\end{split}
\]
for $j = 1, \dots, n$. 
The above equality can be written as
\begin{eqnarray*}
&A&
       \left(
        \begin{array}{@{\,}cccccccc@{\,}}
        \frac{dr_{1}}{ds} \\
        \vdots \\
        \vdots \\
        \vdots \\
        \frac{dr_{n}}{ds}
        \end{array}
\right) = 
 \left(
        \begin{array}{@{\,}cccccccc@{\,}}
        \beta_{1} \\
        \vdots \\
        \vdots \\
        \vdots \\
        \beta_{n}
        \end{array}
\right)\quad\text{where}\\
A& :=& \left(
        \begin{array}{@{\,}cccccccc@{\,}}
\alpha_{1,1} & \alpha_{1,2} & 0 & \ldots & 0 \\
0 & \alpha_{2,2} & \alpha_{2,3} & \ddots & \vdots \\ 
\vdots & \ddots & \ddots & \ddots & \vdots \\
0 & \ldots & 0 & \alpha_{n-1,n-1} & \alpha_{n-1,n} \\
\alpha_{n,1} & 0 & \ldots & 0 & \alpha_{n,n}
\end{array}
       \right).\\
\end{eqnarray*}
Observe that  the above equality says $\frac{dr_j}{ds}(s)$ is independent of $s$.
By Lemma $1$, the determinant of $A$ is positive. 
We first consider the differential $\frac{dr_{1}}{ds}$
and will show that $\frac{dr_{1}}{ds}(0) \ge 0$. 
Put $m=[n/2]$, the largest integer 
such that $m\le n/2$. By the Cramer's formula, the differential $\frac{dr_{1}}{ds}$ of $r_{1}$ is equal to 
\begin{equation*}
\begin{split}
\frac{dr_{1}}{ds}=
&\frac{1}{\det A}\det\left(
        \begin{array}{@{\,}cccccccc@{\,}}
\beta_{1} & \alpha_{1,2} & 0 & \ldots & 0 \\
\vdots & \alpha_{2,2} & \alpha_{2,3} & \ddots & \vdots \\ 
\vdots & 0 & \ddots & \ddots & \vdots \\
\vdots & \vdots & \ddots & \ddots & \alpha_{n-1,n} \\
\beta_{n} & 0 & \ldots & 0 & \alpha_{n,n}
\end{array}
       \right) \\
&=\frac{1}{\det A}\sum_{j=1}^{n}(-1)^{j-1}A_{j-1}\beta_{j}A'_{j+1} \\
&=
\begin{cases}
\frac{1}{\det A}\sum_{k=1}^m (A_{2k-2}\beta_{2k-1}A_{2k}'-A_{2k-1}\beta_{2k}A_{2k+1}'),\, &n=2m\\
\frac{1}{\det A}\sum_{k=1}^m (A_{2k-2}\beta_{2k-1}A_{2k}'-A_{2k-1}\beta_{2k}A_{2k+1}')+A_{n-1}\beta_nA_{n+1}',\, &n=2m+1
\end{cases}
\end{split}
\end{equation*}
where 
\begin{equation*}
A_{j-1} = \begin{cases}                                             
          1 & j = 1 \\                                                         
          \prod_{\ell = 1}^{j-1}\alpha_{\ell, \ell + 1} & j \geq 2            
          \end{cases}, \ \
A'_{j+1} = \begin{cases}  
\prod_{\ell = j}^{n-1}\alpha_{\ell +1, \ell +1} & j \leq n-1 \\
1 & j = n
\end{cases}.
\end{equation*}
We have
\begin{equation*}
\begin{split}
A_{2k-2}\beta_{2k-1}A_{2k}'-A_{2k-1}\beta_{2k}A_{2k+1}'
= A_{2k-2}A_{2k+1}'(\beta_{2k-1}\alpha_{2k,2k} - \alpha_{2k-1,2k}\beta_{2k}). 
\end{split}
\end{equation*}
As $\alpha_{j,j+1}(1,\dots,1)= \beta_{j}$ for $j = 1, \dots, n$, we observe that 
\begin{equation*}
\begin{split}
&\beta_{2k-1}\alpha_{2k,2k}(1, \dots, 1) - \alpha_{2k-1,2k}(1, \dots, 1)\beta_{2k} \\
= &\beta_{2k-1}\{\alpha_{2k,2k}(1, \dots,1) - \beta_{2k}\} \\
= &\beta_{2k-1}\{(1-t)\lvert w_{2k}\rvert^{2b_{2k}}(a_{2k}+2b_{2k}) + a_{2k}t - (1-t)\lvert w_{j+1}\rvert^{2b_{j+1}} - t\}\\
 =&\beta_{2k-1}\{(1-t)\lvert w_{2k}\rvert^{2b_{2k}}(a_{2k}+2b_{2k}-1) +( a_{2k}-1)t\}\geq 0. 
\end{split}
\end{equation*}
The equality holds only if $a_{2k}=1$ and 
 $b_{2k}=0$. Note that $w_i\ne 0$ for any $i=1,\dots, n$ by the assumption.
Anyway we have
\[
\frac{dr_{1}}{ds}(0) \geq  0.
\]
If $n$ is an odd integer,  we see that  $\frac{dr_{1}}{ds}(0)>0$ by the last unpaired term: 
$\frac{dr_{1}}{ds}(0)\ge A_{n-1}\beta_nA_{n+1}'>0$. 
If there exists some $k$ such that $a_{2k}\ge 2$, we have also the strict inequality: $\frac{dr_{1}}{ds}(0)>0$.

Next we consider $\frac{dr_{k}}{ds}$ for $k \geq 2$. 
First observe that our polynomial $f_{II,t}$ has a symmetry for the cyclic permutation of the coordinates
$\si=(1,2,\dots,n)$.
Secondly after cyclic change of coordinates, say
${\bf z}'=(z_1',\dots, z_n')=(z_{\si ^i(1)},\dots, z_{\si ^i(n)})$, the equality (3)  does not change.
That is,  ${\bf w}'(s)=(r_{\si^i(1)}w_{\si ^i(1)},\dots, r_{\si^i (n)} w_{\si^i(n)})$ is the obtained  solution curve.
The tangent vector ${\bf v}' = \frac{d\ww'}{ds}(0)$ is also equal to ${\bf v}$ after 
the corresponding cyclic permutation of coordinates.
Therefore we can apply the above argument to have
the inequality
$\frac{dr_{\si ^i(1)}}{ds}(0)~\ge 0$ for any $i$. 
As we have some $j$ with $a_j\ge 2$, this implies
\[
\frac{dr_{j-1}}{ds}(0)>0.
\]
Now we are ready to show that $\bf{ v} \ne 0$ and ${\bf v}\notin T_{\ww} S_{r_0}^{2n-1}$.
By the assumption of $\ww$, 
the path $\ww(s)$ satisfies 
\begin{eqnarray*}
f_{II, t}(\ww(s), \bar{\ww}(s))& =& (s+1)f_{II, t}(\ww, \bar{\ww}) \equiv 0, \\
\frac{d\| \ww(s) \|^{2}}{ds}\lvert_{s=0} &=& 2\sum_{j=1}^{n}r_{j}(0)\frac{dr_{j}}{ds}(0)\lvert w_{j}\rvert^{2} 
=2\sum_{j=1}^{n}\frac{dr_{j}}{ds}(0)\lvert w_{j}\rvert^{2} > 0. 
\end{eqnarray*}
This implies that $\bf{ v} \ne 0$ and ${\bf v}\notin T_{\ww} S_{r_0}^{2n-1}$.

\vspace{.2cm}
\noindent
{\bf Case 2. } Now we consider the case $\ I_{\ww} \neq \emptyset$. 
Put $I_{\ww}^c$ be the complement of $I_{\ww}$ and  $\Bbb{C}^{*I_{\ww}^{c}} = \{\zz \in \Bbb{C}^{n} \mid z_{i} = 0, i \in I_{\ww} \}$.  We consider the mixed polynomial
$f'(\zz, \bar{\zz}) = f_{II,t}\lvert_{\Bbb{C}^{*I_{\ww}^{c}}}$. 
Let $J$ be the set of indices $j$ for which $z_j$ or $\bar z_j$ appears in $f'$. Note that 
$J\subset I_{\ww}^c$ but it can be a proper subset.

{Case 2-1.} Assume that 
$f' \equiv 0$, i.e., $J=\emptyset$.
We take simply a real analytic path as follows: 
\[
\ww(s) = (s+1)\ww
\]
for $s \in \Bbb{R}$. Since $\ww \in V_{t} \setminus \{ O\}$, we observe that 
\[
f_{II, t}(\ww(s), \bar{\ww}(s)) \equiv 0, \ \ \ \frac{d\| \ww(s) \|^{2}}{ds}\lvert_{s=0} = 2\| \ww \|^{2} > 0. 
\]

{Case 2-2.} Assume that 
$f' \not \equiv 0$.
 Then using the connected components of the graph of~$f'$, we can express $f'$ uniquely as follows.
\[
f'(\zz, \bar{\zz}) = f_{1}(\zz_{I_{1}}) + \cdots f_{k}(\zz_{I_{k}}) 
\]
where the graph of $f_i$ is a bamboo and the variables of $f_i,f_j,\,i\ne j$ are disjoint and the above expression is a join type expression.
Here  $I_{i}$ be the set of indices of variables of $f_{i}$
and $\zz_{I_{i}}=(z_j)_{j\in I_i}$ are the  variables of $f_{i}$ for $i = 1, \dots, k$. 
We have the equality
$\cup_{i=1}^{k}I_{i} =J$ and $I_{i} \cap I_{j} = \emptyset$ for $i \neq j$. 
Put $\Bbb{C}^{I_{i}} = \{\zz \in \Bbb{C}^{n} \mid z_{j} = 0, j \not\in I_{i} \}$.
Fixing $i$, we will construct a curve $\ww_{I_i}(s)$ on $\mathbb C^{*I_i}$ so that 
\[
f_i(\ww_{I_i}(s))=(s+1)f_i(\ww_{I_i}).
\]
The construction of the curve $\ww_{I_i}(s)$ can be reduced to the argument of \cite[Lemma 10]{O3}.  
We will give briefly the  proof which is based on 
 the argument of \cite{O3}. 

For $j\notin J$, we put $w_j(s)=w_j$ and $\ww_{J^c}(s)=\ww_{J^c}\in \mathbb C^{J^c}$ 
where $\Bbb{C}^{J^c} = \{\zz \in \Bbb{C}^{n} \mid z_{j'}~=~0, j' \in J \}$. 
Here $\ww_{J^c}$ is the projection of $\ww$ to $\mathbb C^{J^c}$.
For each $i=1,\dots,k$, we will construct a curve $\ww_{I_i}(s)$ on $\Bbb{C}^{I_{i}}$ 
and  define $\ww_J(s)=\ww_{I_1}(s)+\cdots+\ww_{I_k}(s)$.
Finally we define   a curve $\ww(s)=\ww_{J^c}+\ww_J(s)\in \Bbb{C}^{n}$ so that
\begin{eqnarray*}
f_i(\ww_{I_i}(s))&=&(s+1)f_{i}(\ww_{I_i}),\\
f_{II,t}(\ww(s))&=&f'(\ww(s))\\
&=& f_1(\ww_{I_1}(s))+\cdots+f_k(\ww_{I_k}(s))\\
&=& (s+1)\{f_1(\ww_{I_1})+\cdots+f_k(\ww_{I_k})\}\\
&=&(s+1)f_{II,t}(\ww)\equiv 0 .
\end{eqnarray*}
So we fix $i$.  
For simplicity's sake, 
we  assume $I_{i} = \{j \mid \nu_{i} \leq j \leq \mu_{i}\}$ with $\mu_i\le n$.
The last assumption $\mu_i\le n$ is  for the simplicity of the indices.
This implies that
\[
f_i(\zz_{I_i})=\sum_{j=\nu_i}^{\mu_i-1} z_j^{a_{j}}z_{j+1} \{ |z_j|^{2b_j}(1-t)+t \}.
\]
We will show that there exists a differentiable positive real-valued function solution 
$(r_{\nu_i}(s), \dots, r_{\mu_i}(s))$ of the following equation so that 
$w_j(s)=r_j(s)w_j$, $j\in I_i$ and 
\[
\begin{split}
w_j^{a_{j}}(s)w_{j+1}(s)\{|w_j(s)|^{2b_j}(1-t)+t\}
&=(s+1)w_j^{a_{j}}w_{j+1}\{|w_j|^{2b_j}(1-t)+t\}
\end{split}
\]
for $j=\nu_i,\dots, \mu_i-1$.
We first  consider the equality
\[
\begin{split}
(E''_{j}):\quad \,\,
r_{j}^{a_{j}}\{ \lvert w_{j}\rvert^{2b_{j}}r_{j}^{2b_{j}}(1-t) + t\} &= s_{j}\{ \lvert w_{j}\rvert^{2b_{j}}(1-t) + t\} \\
\text{where}\,\,s_{j} &:= (s+1)/r_{j+1}, \quad\nu_i \leq j\leq \mu_i-1.
\end{split}
\]
First we define $r_{\mu_i}=1$ to start with. 
The left side of  ($E_j''$) is a monotone increasing function of $r_j>0$. Thus assuming $s_j>0$
and considering $s_j$ as an independent variable, we can solve 
 ($E_j''$) in $r_j$  as a function of $s_j$.  Thus we put $r_j=\psi_j(s_j) $.
We claim 
\begin{assertion}
\begin{eqnarray}
 \psi_{j}(1)&=&1, \ \ \frac{d\psi_{j}}{ds_{j}}(s_{j}) > 0,\label{eq7} \\
  \psi_{j}(s_{j})^{a_{j}} &\leq&  s_{j} 
  ,\,\,j = i_{1}, \dots, i_{\ell -1}. \label{eq8}
\end{eqnarray}
\end{assertion}
\begin{proof} 
For $j=\mu_i-1$, the assertion is obvious.  
Assume that $ j<\mu_i-1$.  The assertion (\ref{eq7}) is obvious. 
The assertion (\ref{eq8}) follows from (\ref{eq7}), as  $\psi(s_j)$ is monotone increasing on $s_j$
and 
\[
\lvert w_{j}\rvert^{2b_{j}}r_{j}^{2b_{j}}(1-t) + t\ge \lvert w_{j}\rvert^{2b_{j}}(1-t) + t,\,\,r_j\ge 1.
\]
\end{proof}
 Now we define  $s_{j}(s)$ and $r_{j}(s)$ inductively from  $j=\mu_i$ downward (more precisely from the right end vertex of the graph to the left)
 as follows: 
\begin{eqnarray*}
 r_{\mu_i}(s)=1, \ \
s_{j}(s) = (s+1)/r_{j+1}(s),\, r_{j}(s) = \psi_{j}(s_{j}(s)) 
\end{eqnarray*}
for $\nu_i\le j\le \mu_i-1$. 
\begin{assertion}
$s_{j}(s) \geq 1$ and $r_{j}(s) \geq 1$ for $j = \nu_i, \dots, \mu_i-1$ and $s\geq0$. 
\end{assertion}
\begin{proof}
We show the assertion by a downward induction. 
For $j=\mu_i-1$, the assertion is obvious.  
By the inequality (\ref{eq8}),  we have for $j<\mu_i-1$
\[\begin{split}
s_{j}(s)^{a_{j+1}}&=\left ( \frac{s+1}{r_{j+1}(s)}\right )^{a_{j+1}} = 
\frac{(s+1)^{a_{j+1}}}{\psi_{j+1}(s_{j+1}(s))^{a_{j}+1}}\\
& \geq \frac{(s+1)^{a_{j+1}}}{s_{j+1}(s)} = 
(s+1)^{a_{j+1}-1}r_{j+2}(s)  
\end{split}
\]
for 
$s\ge 0$. 
By the definition of $r_{j}(s)$ and Assertion $3$, 
$s_{j}(s) \geq 1$ and $r_{j}(s) \geq 1$ for $j = \nu_i, \dots, \mu_i-1$ and $s \geq 0$. 
\end{proof}
By Assertion $3$ and Assertion $4$, we see easily  that 
\[\begin{cases}
&\frac{dr_{\mu_i-1}}{ds}(0) > 0\\
&\frac{dr_{j}}{ds}(0) \geq 0,\,\,  j =\nu_i, \dots,\mu_i-2.
\end{cases}
\]
Now we define the  curve $\ww_{I_{i}}(s) $ on $\Bbb{C}^{I_{i}}$  by 
\[
w_{j}(s) = 
          r_{j}(s)w_{j},\quad  j \in I_{i}.                                                   
\]  
As a vector in $\mathbb C^n$, the other coefficients of $\ww_{I_i}(s)$ are defined to be zero.
Then by the construction we have 
\begin{equation*}
\begin{split}
\ww_{I_{i}}(0)=\ww_{I_{i}}, \ \ \ 
&f_{i}(\ww_{I_{i}}(s), \bar{\ww}_{I_{i}}(s)) = (s+1)f_{i}(\ww_{I_{i}}, \bar{\ww}_{I_{i}}), \\
&\frac{d\| \ww_{I_{i}}\|^{2}}{ds}\lvert_{s=0} \geq 2\frac{dr_{\mu_{i}-1}}{ds}(0)\lvert w_{\mu_{i}-1}\rvert^{2} > 0 
\end{split}
\end{equation*}
where $\lvert s \rvert \ll 1$ and $1 \leq i \leq k$. 
After  constructing $\ww_{I_i}(s)$ for each $ i=1,\dots, k$, we define a~smooth curve
$\ww(s)=(w_1(s),\dots, w_n(s))$ by the summation
\[\begin{split}
\ww(s)&=\ww_{J^c}(s)+\ww_J(s),\\
\ww_J(s)&=\ww_{I_1}(s)+\cdots+\ww_{I_k}(s).
\end{split}
\]
Then $\ww(s)$ satisfies 
\begin{equation*}
\begin{split}
f_{II,t}(\ww(s))&=
f'(\ww(s), \bar{\ww}(s))
= \sum_{i=1}^{k}f_{i}(\ww_{I_{i}}(s), \bar{\ww}_{I_{i}}(s))\\
& = (s+1)\sum_{i=1}^{k}f_{i}(\ww_{I_{i}}, \bar{\ww}_{I_{i}}) = (s+1)f'(\ww, \bar{\ww})\\
&=(s+1)f_{II,t}(\ww,\bar\ww)\equiv 0, \\
\frac{d\| \ww(s)\|^{2}}{ds}\lvert_{s=0} &= \sum_{i=1}^{k}\frac{d\| \ww_{I_{i}}(s)\|^{2}}{ds}\lvert_{s=0} > 0. 
\end{split} 
\end{equation*}
Thus defining $\mathbf{v} := \frac{d \ww}{ds}(0)$, we conclude  
$\mathbf{v}\in T_{\ww}V_{t}\setminus T_{\ww}S^{2n-1}_{r_{0}}$. 
This completes the proof of the transversality.
\begin{remark}
In the above argument,  if $v_n$ is a vertex of the graph of $f_i$ and it is not the right end vertex,
we  use the expression
 $I_{i} = \{j \mod\, n\mid \nu_{i} \leq j \leq \mu_{i}\}$ with $\mu_i> n$.
This implies that
\[
f_i(\zz_{I_i})=\sum_{j=\nu_i}^{\mu_i-1} z_j^{a_{j}}z_{j+1} \{ |z_j|^{2b_j}(1-t)+t \}
\]
where $z_{j+n}=z_j, a_{j+n}=a_j, b_{j+n}=b_j$.
We do the same argument as above starting the right end variable $z_{\mu_i}=z_{\mu_i-n}$.
\end{remark}
\subsection{Applications}

\begin{corollary} \label{Cor1}
Let $V_t$ be the hypersurface defined by $f_{II,t}$ and let $K_{t,r}$ be its link.
Then there exists an isotopy 
$\psi_{t} : (S^{2n-1}_{r}, K_{0,r}) \rightarrow (S^{2n-1}_{r}, K_{t,r})$ for $0\le t\le 1$ with $\psi_0=\id$. 
\end{corollary}
This is immediate from  Ehresmann's fibration theorem (\cite{W}). 
As for the Milnor fibration of the second type, we have:
\begin{corollary} 
For a fixed $r>0$, 
there exists a positive real number~$\eta_{0}$ so that 
$f_{II,t}^{-1}(\eta)$ and $S^{2n-1}_{r}$ intersect transversely for any $\eta$, $\lvert \eta \rvert \leq \eta_{0}$ and $0\le t\le 1$.  
In particular this implies that 
there exists a  family of diffeomorphisms $\psi_t:  \partial E_{0}(\eta_{0}, r)\to  \partial E_{t}(\eta_{0}, r)$ such that 
the following diagram is commutative: 
\def\mapright#1{\smash{\mathop{\longrightarrow}\limits^{{#1}}}}
\def\mapdown#1{\Big\downarrow\rlap{$\vcenter{\hbox{$#1$}}$}}
\[\begin{matrix}
\partial E_{0}(\eta_{0}, r)&\mapright{\psi_t}& \partial E_{t}(\eta_{0}, r) \\
\mapdown{f_{II,0}}&&\mapdown{f_{II,t}}\\
S^{1}_{\eta_{0}}&=& S^{1}_{\eta_{0}} 
\end{matrix}\]
where $\partial E_{t}(\eta_{0}, r) = \{\zz \in \Bbb{C}^{n} \mid \lvert f_{II, t}(\zz)\rvert = \eta_{0}, \| \zz \| \leq r \}$. 
\end{corollary}
\begin{proof}
Fix a positive real number $r$. Let 
\[\begin{split}
\partial \mathcal{E}(\eta_{0}, r)& := \{(\zz, t) \in \Bbb{C}^{n} \times [0, 1] \mid \lvert f_{II,t}(\zz)\rvert = \eta_{0}, \| \zz \| \leq r \}\\
\partial ^2\mathcal{E}(\eta_{0}, r)& := \{(\zz, t) \in \Bbb{C}^{n}\times [0, 1] \mid \lvert f_{II,t}(\zz)\rvert = \eta_{0}, \| \zz \| = r \}.
\end{split}
\]
Since $S^{2n-1}_{r}$ intersects with $V_{t}$ transversely and 
$S^{2n-1}_{r} \cap V_{t}$ is compact for any $0 \leq t \leq 1$, 
there exists a positive real number $\eta_{0}$ such that 
$f_{II,t}^{-1}(\eta)$ and $S^{2n-1}_{r}$ intersect transversely for any $\eta, \lvert \eta \rvert \leq \eta_{0}$ 
and $0 \leq t \leq 1$. 
Thus the projection 
$\pi' :( \partial \mathcal{E}(\eta_{0}, r),\partial^2\mathcal{E}(\eta_{0}, r)) \rightarrow [0, 1]$ 
is a proper submersion. 
By the Ehresmann's fibration theorem \cite{W}, 
$\pi'$ is a locally trivial fibration over $[0, 1]$. 
So the projection 
$\pi'$ induces 
a family of isomorphisms $\psi_{t} : \partial E_{0}(\eta_{0}, r) \rightarrow \partial E_{t}(\eta_{0}, r)$ 
of fibrations 
for any $\zz$ with $\lvert f_{II,t}(\zz) \rvert \leq \eta_{0}$ and $0 \leq t \leq 1$. 
\end{proof}
Now we consider again Milnor fibration of the link complement. Consider the mapping
\begin{eqnarray}\label{spherical-Milnor}
f_{II,t}/|f_{II,t}|:\,\, S_r^{2n-1}\setminus K_{t,r}\to S^1.
\end{eqnarray}
As $f_{II,t}(\zz,\bar\zz)$ is polar weighted homogeneous polynomial,  the $S^1$-action gives non-vanishing 
vector field, denoted as $\frac{\partial}{\partial \theta}$ on $S_r^{2n-1}\setminus K_{t,r}$ so that 
$f_{II,t}(s\circ \zz)=s^{d_p}f_{II,t}(\zz)$ for $s\in S^1$, 
this gives fibration structure for (\ref{spherical-Milnor})  for any $r>0$ and we call it 
{\em a spherical Milnor fibration or a Milnor fibration of the first description.}
The isomorphism class of the fibration does not depend on $r$.
Consider  two fibrations 
\[
f_{II, t} : \partial E_{t}(\eta_{0}, r) \rightarrow S^{1}_{\eta_{0}}, \ \ \
f_{II, t}/\lvert f_{II, t}\rvert : S^{2n-1}_{r} \setminus K_{t, r} \rightarrow S^{1}.
\]
The first fibration is called {\em a Milnor fibration of the second description or a tubular Milnor fibration.}
The isomorphism class of the tubular fibration does not depend on the choice of $r$ and $\eta_0\ll r$. 
As we know that two fibrations
are isomorphic for sufficiently small $r > 0$ and any $t$ (\cite[Theorem 36]{O2}), they are isomorphic for any $r$. 
Combining this and Corollary 2, we can sharpen Corollary \ref{Cor1} as follows.
\begin{corollary}
Let $\psi_{t} : (S^{2n-1}_{r}, K_{0,r}) \rightarrow (S^{2n-1}_{r}, K_{t,r})$ be an isotopy in Corollary 1.
$\psi_t$ can be constructed so that 
 the following diagram is commutative.
\def\mapright#1{\smash{\mathop{\longrightarrow}\limits^{{#1}}}}
\def\mapdown#1{\Big\downarrow\rlap{$\vcenter{\hbox{$#1$}}$}}
\[\begin{matrix}
S^{2n-1}_{r}\setminus K_{0,r} & \mapright{\psi_{t}} & S^{2n-1}_{r}\setminus  K_{t,r}\\
\mapdown{f_{II,0}/\lvert f_{II,0}\rvert} && \mapdown{f_{II,t}/\lvert f_{II,t}\rvert}\\
S^{1} & \mapright{id} & S^{1}
\end{matrix}\]
\end{corollary}
Taking $t=1$,  we get  a positive answer to the 
conjecture in \cite{O3}. 
\begin{proof}
Choose a positive real number $\eta_0$ as in Corollary 2.
Consider the cobordism variety
$\mathcal V_r:=\{(\zz,t)\in S_r^{2n-1}\times [0,1]\,|\, f_{II,t}(\zz,\bar\zz)=0\}$
and its open neighborhood
$\mathcal W_{\eta}:=\{(\zz,t)\in S_r^{2n-1}\times [0,1]\,|\, |f_{II,t}(\zz)|< \eta\}$ of $\mathcal V_t$.
Consider  the projection mapping 
\[
\pi:S_r^{2n-1}\times [0,1]\to [0,1], \quad (\zz,t)\mapsto t.
\]
Let $\frac{\partial}{\partial \theta}'$ be the projection of the gradient vector of 
$\Im \log\,f_{II,t}(\zz,\bar\zz)$
to the tangent space of
$S_r^{2n-1}\times [0,1]\setminus \mathcal V_t$.
Using the vector field $\frac{\partial}{\partial \theta}$ on $S_r^{2n-1}\times [0,1]$, we see easily that 
$\frac{\partial}{\partial \theta}'$ is a~non-vanishing vector on 
$S_r^{2n-1}\times [0,1]\setminus \mathcal V_t$ 
which is linearly independent with $\frac{\partial}{\partial t}$ over~$\Bbb{R}$. Now 
we construct 
a vector filed $\mathcal X$ on $S_r^{2n-1}\times [0,1]\setminus \mathcal V_t$ such that

\begin{enumerate}
\item $d\pi_*(\mathcal X(\zz,t))=\frac{\partial}{\partial t}$ and 
$\{\mathcal X(\zz,t),\frac{\partial}{\partial \theta}'(\zz,t)\}$
are orthogonal.
\item For $(\zz,t)\in \mathcal W_{\eta_0/2}$, $\{\mathcal X(\zz,t),\,\,{\rm grad}\,|f_{II,t}|(\zz,t)\}$
are also orthogonal. 
\end{enumerate} 
The condition (1) implies the argument of $f_{II,t}$ does not change along the integral curve of $\mathcal X$. 
The conditions (1) and  (2) implies the integral curve of 
$\mathcal X$ keeps the level $f_{II,t}=\eta$ for any $\eta$ with $|\eta|\le \eta_0/2$. 
Thus integral curves of vector field $\mathcal X$ exists over $[0,1]$ and 
we construct  the isotopy $\psi_t$  using the  integration curves of $\mathcal X$.
\end{proof}
\begin{remark}
Let $f(\zz, \bar{\zz}) = \sum_{i = 1}^{m} c_{i}\zz^{\nu_{i}}\bar{\zz}^{\mu_{i}}$ be a full  simplicial mixed polynomial 
and $g(\zz)$ be the associated Laurent polynomial of $f$. 
The last author  defined a canonical diffeomorphism of $\varphi : \Bbb{C}^{*n} \rightarrow \Bbb{C}^{*n}$ as follows 
$($\cite{O1}$)$: 
\begin{equation*}
\begin{split}
\varphi : \Bbb{C}^{*n} &\rightarrow \Bbb{C}^{*n}, \\
\zz = (\rho_{1}\exp(i\theta_{1}), \dots, \rho_{n}\exp(i\theta_{n})) 
&\mapsto \ww = (\xi_{1}\exp(i\theta_{1}), \dots, \xi_{n}\exp(i\theta_{n}))
\end{split}
\end{equation*}
where $(\rho_{1},\dots,\rho_{n})$ and $(\xi_{1},\dots,\xi_{n})$ satisfy 
\[
(N + M)\begin{pmatrix} \log \rho_{1} \\ \vdots \\ \log \rho_{n} \end{pmatrix} 
= (N - M)\begin{pmatrix} \log \xi_{1} \\ \vdots \\ \log \xi_{n} \end{pmatrix} 
\]
where $N = (\nu_{1}, \dots, \nu_{n})$ and $M = (\mu_{1}, \dots, \mu_{n})$. 
Then $\varphi$ satisfies that 
$\varphi(\Bbb{C}^{*n} \cap f^{-1}(c)) = \Bbb{C}^{*n} \cap g^{-1}(c)$ for any $c\in \mathbb C$ 
$($\cite[Theorem 10]{O1}$)$. 
However $\varphi$ cannot be extended to a homeomorphism of $\Bbb{C}^{n}\setminus\{ O\}$ to itself in general, 
except the case of mixed Brieskorn polynomial. 
\end{remark}

\begin{Example} We will give an example of the above remark.
Let $f(\zz, \bar{\zz})$ be a simplicial polynomial defined by
\[
f(\zz, \bar{\zz}) = z_{1}^{3}\bar{z}_{1}z_{2} + z_{2}^{3}\bar{z}_{2}z_{3} + z_{3}^{3}\bar{z}_{3}z_{1}.
\] 
Then the diffeomorphism of $\varphi : \Bbb{C}^{*3} \rightarrow \Bbb{C}^{*3},
\zz = (z_{1}, z_{2}, z_{3}) \mapsto \ww = (w_{1}, w_{2}, w_{3})$ is given by 
\[
\begin{pmatrix} w_{1} \\ w_{2} \\ w_{3} \end{pmatrix} = 
\begin{pmatrix}
 \lvert z_{1}\rvert^{\frac{17}{9}}\lvert z_{2}\rvert^{\frac{-4}{9}}\lvert z_{3}\rvert^{\frac{2}{9}}\exp(i\theta_{1}) \\ 
 \lvert z_{1}\rvert^{\frac{2}{9}}\lvert z_{2}\rvert^{\frac{17}{9}}\lvert z_{3}\rvert^{\frac{-4}{9}}\exp(i\theta_{2}) \\
 \lvert z_{1}\rvert^{\frac{-4}{9}}\lvert z_{2}\rvert^{\frac{2}{9}}\lvert z_{3}\rvert^{\frac{17}{9}}\exp(i\theta_{3}) 
\end{pmatrix}. 
\]
The above map cannot extend to a continuous map on  the coordinate planes $\{(z_{1},z_{2}, z_{3}) \in \Bbb{C}^3 \mid z_{1}z_{2}z_{3}=0 \}$ as the negative exponents in the above description.
So the map $\varphi$ cannot extend to a homeomorphism of $\Bbb{C}^{3}\setminus\{ O\}$ to itself.

\end{Example}



\end{document}